\documentclass[a4paper,11pt]{article}%

\pdfoutput=1 
\usepackage{amsmath}%
\usepackage{amsfonts}%
\usepackage{amssymb}%
\usepackage{amsthm}
\usepackage{stmaryrd}
\usepackage{graphicx}
\usepackage[utf8]{inputenc}
\usepackage[T1]{fontenc}
\usepackage[french, english]{babel}
\usepackage{enumitem}
\usepackage{tikz}
\usetikzlibrary{decorations.markings}
\usepackage[colorlinks=true]{hyperref}

\newtheorem{theorem}{Theorem}

\newtheorem{claim}[theorem]{Claim}

\newtheorem{lemma}[theorem]{Lemma}

\theoremstyle{definition}

\theoremstyle{remark}

\newtheorem{remark}[theorem]{Remark}

\newlength{\espaceavantspecialthm}
\newlength{\espaceapresspecialthm}
\newcommand{\R}{\mathbb{R}}

{\\}

\newenvironment{defi*}[1][]{
\vskip \espaceavantspecialthm \noindent \textbf{D\'efinition.} }%
{\vskip \espaceapresspecialthm}

\counterwithin{equation}{section}

\tikzset{->-/.style={decoration={
  markings,
  mark=at position .5 with {\arrow{latex}}},postaction={decorate}}}

\newcommand\test[1]{
\pgfmathsetmacro{\var}{#1}
\pgfmathparse{ifthenelse(\var>=0,"positif","négatif")} \pgfmathresult}%

\newcommand{\hgline}[3]{
\pgfmathsetmacro{\thetaone}{#1}
\pgfmathsetmacro{\thetatwo}{#2}
\pgfmathsetmacro{\theta}{(\thetaone+\thetatwo)/2}
\pgfmathsetmacro{\phi}{abs(\thetaone-\thetatwo)/2}
\pgfmathsetmacro{\close}{less(abs(\phi-90),0.0001)}
\ifdim \close pt = 1pt
    \draw[->-, color=#3] (\thetaone:1) -- (\thetatwo:1);
\else
	\pgfmathsetmacro{\R}{tan(\phi)}
	\pgfmathsetmacro{\test}{(\thetaone-\thetatwo)/abs(\thetaone-\thetatwo)}
	\draw[->-, color=#3] (\thetaone:1) arc (\thetaone+\test*90:\thetaone+\test*(270-2*\phi):\R);
\fi
}

\newcommand{\hglinefill}[3]{
\pgfmathsetmacro{\thetaone}{#1}
\pgfmathsetmacro{\thetatwo}{#2}
\pgfmathsetmacro{\theta}{(\thetaone+\thetatwo)/2}
\pgfmathsetmacro{\phi}{abs(\thetaone-\thetatwo)/2}
\pgfmathsetmacro{\close}{less(abs(\phi-90),0.0001)}
\ifdim \close pt = 1pt
    \filldraw[->-, color=#3] (\thetaone:1) -- (\thetatwo:1) ;
\else
	\pgfmathsetmacro{\R}{tan(\phi)}
	\pgfmathsetmacro{\test}{(\thetaone-\thetatwo)/abs(\thetaone-\thetatwo)}
	\filldraw[color=#3, opacity=.2] (\thetaone:1) arc (\thetaone+\test*90:\thetaone+\test*(270-2*\phi):\R) arc (\thetaone+\test*(270-2*\phi)-90:\thetaone+\test*90+90:1);
\fi
}

\begin{document}

\sloppy

\title{Generalized foliations for homeomorphisms isotopic to a pseudo-Anosov homeomorphism: a geometric realization of a result by Fathi}
\author{Emmanuel Militon}
\maketitle

\begin{abstract}
We give a new proof of a result by Fathi, which states that, to any homeomorphism of a closed surface which is isotopic to a pseudo-Anosov homeomorphism, we can associate a stable and an unstable invariant partition of the surface with properties which are similar to the unstable and the stable foliation of a pseudo-Anosov homeomorphism. 
\end{abstract}

\selectlanguage{english}

\section{Introduction}

Pseudo-Anosov homeomorphisms of surfaces are a generalization of Anosov homeomorphisms to higher-genus surfaces which play an important role in the Nielsen-Thurston classification of isotopy classes of homeomorphisms of surfaces. Indeed, by this classification, any homeomorphism of a surface which does not preserve the isotopy class of any finite collection of disjoint essential simple closed curves has to be isotopic to a pseudo-Anosov homeomorphism. 

The dynamics of a homeomorphism which is isotopic to a pseudo-Anosov homeomorphism is rather well-understood. Indeed, Handel proved that any such homeomorphism possesses an invariant closed subset on which it is semiconjugate to its pseudo-Anosov representative (see \cite{MR0805836}). In this note, we are interested in a more precise result by Fathi (see Theorem 2.6 in \cite{MR1062759}), which states that such homeomorphisms have also a stable and an unstable partition which are similar to that of its pseudo-Anosov representative. However, Fathi's approach is rather abstract and does not give an explicit construction of those partitions. In this note, we propose an explicit construction of those stable and unstable partitions. We obtain this partition by shadowing the orbits of the pseudo-Anosov homeomorphism only in the future or only in the past. In the last section, we provide an extension of this theorem in the case of a surface with marked points. This extension allows to recover some of the results of a recent preprint by Alejo García-Sassi and Fábio Armando Tal \cite{GT}.

I would like to thank Jonathan Bowden, Alejo García-Sassi, Pierre-Antoine Guihéneuf, Sebastian Hensel and Frédéric Le Roux for helpful discussions about this note and Pierre-Antoine Guihéneuf for his comments on a first version of this note. Finally, I am very thankful to Nastaran Einabadi, Federica Fanoni, Sebatian Hensel, Frédéric Le Roux, Nelson Schuback for finding an important flaw and especially to Frédéric Le Roux for indicating many smaller mistakes in a previous version of the article.

\section{Notation}

Let $S$ be a closed surface with genus $\geq 2$. 
Let $A$ be a pseudo-Anosov homeomorphism of $S$ and denote by $\tilde{A}$ one of its lift to the universal cover $\tilde{S}$ of $S$. We refer to \cite{MR2850125} and \cite{MR0568308} for a precise definition of pseudo-Anosov homeomorphisms and more context. We identify $\tilde{S}$ with the hyperbolic plane $\mathbb{H}^2$. By definition, the homeomorphism $A$ preserves two transverse singular foliations, the stable one $\mathcal{F}^s$ and the unstable one $\mathcal{F}^u$. Let us denote $\tilde{\mathcal{F}}^s$ and $\tilde{\mathcal{F}}^u$ the lifts of those foliations to $\tilde{S}$. For $i=s,u$, we call leaf of $\tilde{\mathcal{F}}^i$ either a regular leaf of $\tilde{\mathcal{F}}^i$ or the union of a singularity with all the branches of $\tilde{\mathcal{F}}^i$ which converge to this singularity.

The foliations $\tilde{\mathcal{F}}^s$ and $\tilde{\mathcal{F}}^u$ are endowed with respective transverse measures $\mu^u$ and $\mu^s$ with the following property. There exists $0<\lambda <1$ such that, for any arc $\tilde{\alpha} : [0,1] \rightarrow \tilde{S}$ and any $n \geq 0$,
$$ \left\{
\begin{array}{rcl}
\mu^u(\tilde{A}^{-n}(\tilde{\alpha})) & = & \lambda^n \mu^u(\tilde{\alpha}) \\
\mu^s(\tilde{A}^{n}(\tilde{\alpha})) & = & \lambda^n \mu^s(\tilde{\alpha}).
\end{array} \right.$$
Moreover, the transverse measures $\mu^u$ and $\mu^s$ are invariant under the action of the group $\pi_1(S)$ of deck transformations of $\tilde{S}$.

We can then define pseudo-distances $d_u$ and $d_s$ on $\tilde{S}$ in the following way. For $i=s,u$, and $\tilde{x},\tilde{y} \in \tilde{S}$, $d_i(\tilde{x},\tilde{y})$ is the infimum of $\mu^i(\tilde{\alpha})$ over all arcs $\tilde{\alpha}:[0,1] \rightarrow \tilde{S}$ with $\tilde{\alpha}(0)=\tilde{x}$ and $\tilde{\alpha}(1)=\tilde{y}$. Observe that, for $\tilde{x},\tilde{y} \in \tilde{S}$, $d_{u}(\tilde{x},\tilde{y})=0$ if and only if $\tilde{x}$ and $\tilde{y}$ belong to the same leaf of $\tilde{\mathcal{F}}^s$ and a similar result holds for $d_s$. The deck transformations of $\tilde{S}$ are isometries for both pseudo-distances $d_u$ and $d_s$. Finally recall that $d=d_s+d_u$ defines a distance on $\tilde{S}$ which is compatible with the topology of $\tilde{S}$. Finally, due to the properties of $\mu^u$ and $\mu^s$, for any two points $\tilde{x}$ and $\tilde{y}$ of $\tilde{S}$,
$$d_u(\tilde{A}^{-1}(\tilde{x}),\tilde{A}^{-1}(\tilde{y}))=\lambda d_u(\tilde{x},\tilde{y})$$
and
$$d_s(\tilde{A}(\tilde{x}),\tilde{A}(\tilde{y}))=\lambda d_s(\tilde{x},\tilde{y}).$$

In what follows, we will need to compactify the universal cover $\tilde{S}$ with a circle at infinity which we denote by $\partial \tilde{S}$. When we see $\tilde{S}=\mathbb{H}^2$ as the Poincar\'e disk, which is the unit open disk, the circle at infinity is just the unit circle. We set $\overline{S}=\tilde{S} \cup\partial \tilde{S} $.

This compactification satisfies the following properties that we will use.
\begin{enumerate}
\item Any end of a leaf in $\tilde{\mathcal{F}}^s$ or in $\tilde{\mathcal{F}}^u$ converges to a point of $\partial \tilde{S}$.
\item For $i=s,u$, no point of $\partial \tilde{S}$ is the end of two different leaves of $\tilde{\mathcal{F}}^i$.
\item Any homeomorphism $\tilde{g}$ of $\tilde{S}$ which is a lift of a homeomorphism $g$ of $S$ extends continuously to $\overline{S}$. By abuse of notation, we will still denote by $\tilde{g}$ this continuous extension.
\item Given two isotopic homeomorphisms $f$ and $g$ of $S$, their isotopic lifts $\tilde{f}$ and $\tilde{g}$ are equal on the boundary at infinity $\partial \tilde{S}$. 

\end{enumerate}

For any subset $A$ of $\tilde{S}$, we denote by $\overline{A}$ its closure in $\overline{S}$. In the case where $A=L$ is a regular leaf, it consists of $L$ with two extra points on $\partial \tilde{S}$. In the case of a leaf $L$ which contains a $k$-pronged singularity, it consists of $L$ with $k$ extra point on $\partial \tilde{S}$.

For $C>0$, and $L \in \tilde{\mathcal{F}}^s$, we define
$$L_C= \left\{ \tilde{x} \in \tilde{S} \ | \ d_u(\tilde{x},L) \leq C \right\}.$$
We define similarly $L_C$ for $L \in \tilde{\mathcal{F}}^u$.

The distance $d_u$ induces a distance on the space $\mathcal{L}^s$ of leaves of the foliation $\tilde{\mathcal{F}}^s$. However, this metric space is not necessarily complete. We will denote by $\overline{\mathcal{L}}^s$ the completion of $(\mathcal{L}^s,d_u)$. We will see this completion as the set of equivalence classes of Cauchy sequences $(L_n)_n$ of elements of  $\mathcal{L}^s$ for the equivalence relation defined by $(L_n)_n \sim (L'_n)_n$ if and only if
$$ \lim_{n \rightarrow +\infty} d_{u}(L_n,L'_n)=0.$$
Constant sequences define an embedding $\mathcal{L}^s \subset \overline{\mathcal{L}^s}$ and the distance $d_u$ on $\mathcal{L}^s$ extends to a distance $d_u$ on $\overline{\mathcal{L}^s}$. For any Cauchy sequence $(L_n)$ of leaves of $\mathcal{L}^s$, the sequence $(\overline{L}_n)$ of  compact subsets of $\overline{S}$ converges for the Hausdorff topology to a connected compact subset of $\overline{S}$. When this sequence $(L_n)$ is equivalent to a constant sequence, then the sequence of subset is converging to the closure of the leaf of the constant sequence. Otherwise, this sequence will converge to one point on the boundary at infinity. This limit point does not change if we change the sequence $(L_n)$ to an equivalent one. Hence, to any point $\xi$ $\overline{\mathcal{L}}^s$, we associate a subset $\overline{L}(\xi)$ which is either the closure of a stable leaf or one point on $\partial \tilde{S}$

Similarly, we define $\overline{\mathcal{L}}^u$ as the completion of $(\mathcal{L}^u,d_s)$, where $\mathcal{L}^u$ is the space of leaves of the foliation $\tilde{\mathcal{F}}^u$.

Observe that the lift $\tilde{A}$ of the pseudo-Anosov homeomorphism as well as any deck transformation in $\pi_1(S)$ act on the complete spaces $\overline{\mathcal{L}}^s$ and $\overline{\mathcal{L}}^u$ by homeomorphisms.

\section{The generalized foliation}

Let $f$ be a homeomorphism which is isotopic to $A$. Take an isotopy $(f_t)$ with $f_1=f$ and $f_0=A$ which pointwise fixes $F$ and lift it to an isotopy $(\tilde{f}_t)_{t \in [0,1]}$ of $\tilde{S}$ with $\tilde{f}_0=\tilde{A}$. Then $\tilde{f}=\tilde{f}_1:\tilde{S} \rightarrow \tilde{S}$ is a lift of $f$, which is equal to $\tilde{A}$ on the boundary at infinity $\partial \tilde{S}$.

\begin{theorem} \label{Th:generalized}
Fix $i=s,u$. To each point $\xi$ of $\overline{\mathcal{L}}^i$, we associate a compact connected subset $\overline{\eta}^i(\xi)$ of $\overline{S}$ with the following properties.
\begin{enumerate}
\item The $\eta^{i}(\xi)=\overline{\eta}^i(\xi) \cap \tilde{S}$, for $\xi \in \overline{\mathcal{L}}^i$, are pairwise disjoint and cover $\tilde{S}$. Moreover, if $L \in \mathcal{L}^i$, then $\eta^i(L) \neq \emptyset$.
\item For every $L \in \mathcal{L}^i$, the set $\overline{S} \setminus \overline{\eta}^{i}(L)$ has the same number of connected components as $\overline{S} \setminus \overline{L}$. For every $\xi \in \overline{\mathcal{L}}^i \setminus \mathcal{L}^i$, the set $\overline{S} \setminus \overline{\eta}^{i}(\xi)$ has only one connected component.
\item For every $\xi \in \overline{\mathcal{L}}^i$, 
$$ \tilde{f} \left( \overline{\eta}^i(\xi) \right)=\overline{\eta}^i \left(\tilde{A}(\xi) \right).$$
\item For every $\xi \in \overline{\mathcal{L}}^i$ and any deck transformation $\gamma \in \pi_1(S)$,
$$ \gamma \left( \overline{\eta}^i(\xi) \right)=\overline{\eta}^i \left( \gamma(\xi) \right).$$
\item For every $\xi \in \overline{\mathcal{L}}^i$, $\overline{\eta}^i(\xi) \cap \partial \tilde{S}= \overline{L}(\xi) \cap \partial \tilde{S}$.
\item The sets $\eta^s(\xi)$ are the stable sets of all its points, namely for every $\xi \in \overline{\mathcal{L}}^{s}$ and $\tilde{x} \in \eta^{s}(\xi)$,
$$ \eta^{s}(\xi)= \left\{ \tilde{y} \in \tilde{S} \ | \ \exists C'>0, \ \forall n \geq 0, d(\tilde{f}^{n}(\tilde{x}),\tilde{f}^{n}(\tilde{y})) \leq C' \right\}.$$
The sets $\eta^u(\xi)$ are the unstable sets of all its points, namely for every $\xi \in \overline{\mathcal{L}}^{u}$ and $\tilde{x} \in \eta^{u}(\xi)$,
$$ \eta^{u}(L)= \left\{ \tilde{y} \in \tilde{S} \ | \ \exists C'>0, \ \forall n \geq 0, d(\tilde{f}^{-n}(\tilde{x}),\tilde{f}^{-n}(\tilde{y})) \leq C' \right\}.$$
\item For every point $\xi$ of $\overline{\mathcal{L}^i}$ such that $\eta^{i}(\xi) \neq \emptyset$ and every neighbourhood $U$ of $\overline{\eta}^i(\xi)$ in $\overline{S}$, there exists $\eta>0$ such that, for every point $\xi'$ of $\overline{\mathcal{L}}^i$ such that $d_u(\xi,\xi') \leq \eta$, we have $\overline{\eta}^i(\xi') \subset U$.
\end{enumerate}
\end{theorem}

Fix $i=s,u$. By Theorem \ref{Th:generalized}, the sets $\eta^{i}(\xi)$, with $\xi \in \overline{\mathcal{L}}^i$, form a $\tilde{f}$-invariant partition of $\tilde{S}$ which projects to a $f$-invariant partition of $S$. 

\begin{remark}
Point 6. of Theorem \ref{Th:generalized} proves that the partitions given by Theorem \ref{Th:generalized} are unique.
\end{remark}

\begin{remark} To each point $\tilde{x}$ of $\tilde{S}$, one can associate a unique point $\xi_s \in \overline{\mathcal{L}}^s$ and a unique point $\xi_u \in \overline{\mathcal{L}}^u$ such that $\tilde{x} \in \eta^{s}(\xi_s) \cap \eta^u(\xi_u)$. Hence the above theorem defines a map $\tilde{S} \rightarrow \overline{\mathcal{L}}^s \times \overline{\mathcal{L}}^u$. We prove now that this map is continuous.

\begin{proof}
By point 7. in Theorem \ref{Th:generalized}, the map, which to any pair of points $(\xi_{s},\xi_u) \in \overline{\mathcal{L}}^s \times \overline{\mathcal{L}}^u$ associates the subsets $(\overline{\eta}^{s}(\xi_s), \overline{\eta}^{u}(\xi_u))$ is continuous for a suitable topology on the spaces of partitions given by $\eta^{s}$ and $\eta^u$. Moreover, we can endow the partitions given by $\overline{\eta}^s$ and $\overline{\eta}^u$ with distances $d'_s$ and $d'_u$ which are compatible with the topology so that the maps $\xi_s \mapsto \overline{\eta}^{s}(\xi_s)$ and $\xi_u \mapsto \overline{\eta}^{u}(\xi_u)$ are isometries. Hence the map 
$$(\xi_s,\xi_u) \mapsto (\overline{\eta}^{s}(\xi_s), \overline{\eta}^{u}(\xi_u))$$
is a homeomorphism. Now the map which, to a point $\tilde{x}$ of $\tilde{S}$ associates the unique pair $(\overline{\eta}^{s}(\xi_s), \overline{\eta}^{u}(\xi_{u}))$ such that $\tilde{x} \in \overline{\eta}^{s}(\xi_s) \cap \overline{\eta}^{u}(\xi_u)$ is continuous, which proves the continuity of the map $\tilde{S} \rightarrow \overline{\mathcal{L}}^s \times \overline{\mathcal{L}}^u$. 
\end{proof}
This map semi-conjugates the homeomorphism $\tilde{f}$ to the restriction of the action of $\tilde{A}$ on $\overline{\mathcal{L}}^s \times \overline{\mathcal{L}}^u$ to an invariant subset. This makes the connection with Fathi's result (see Theorem 2.6 in \cite{MR1062759}).

\end{remark}

\begin{proof}[Proof of Theorem \ref{Th:generalized}]
We will prove this theorem in the case $i=s$. The case $i=u$ is made by exchanging the roles of $\tilde{f}$ and $\tilde{f}^{-1}$, of $\tilde{A}$ with $\tilde{A}^{-1}$, of $u$ and $s$ in the proof below.

The starting point of the proof is the following claim.

\begin{claim} \label{uniformbound}
There exists $C_0>0$ such that, for any $\tilde{x} \in \tilde{S}$,
$$\left\{ \begin{array}{rcl}
d(\tilde{f}(\tilde{x}),\tilde{A}(\tilde{x})) & \leq & C_0 \\
d(\tilde{f}^{-1}(\tilde{x}),\tilde{A}^{-1}(\tilde{x})) & \leq & C_0
\end{array}
\right.$$
\end{claim}

\begin{proof}

Let $D \subset \tilde{S}$ be a compact subset with $\pi(D)=S$, where $\pi:\tilde{S} \rightarrow S$ is the projection. By compactness
$$C'_0= \max_{\tilde{x} \in D} d(\tilde{f}(\tilde{x}),\tilde{A}(\tilde{x}))$$
exists. However, by construction of $\tilde{f}$, for any deck transformation $\gamma \in \pi_1(S)$, there exists a deck transformation $\xi(\gamma)$ such that
$$\left\{ \begin{array}{rcl}
 \tilde{f} \circ \gamma & = & \xi(\gamma) \circ \tilde{f} \\
   \tilde{A} \circ \gamma & = & \xi(\gamma) \circ \tilde{A}.
\end{array} \right.$$
Indeed, denote by $\xi(\gamma)$ the deck transformation such that $\tilde{A} \circ \gamma = \xi(\gamma) \circ \tilde{A}$. Then the isotopies $(\tilde{f}_t \circ \gamma)_t$ and $(\xi(\gamma) \circ \tilde{f}_t)_t$ both lift the isotopy $(f_t)_t$ and are equal for $t=0$. Hence, for any $t$ , $\tilde{f}_t \circ \gamma=\xi(\gamma) \circ \tilde{f}_t$ and, in particular, $\tilde{f} \circ \gamma=\xi(\gamma) \circ \tilde{f}$.

As $d$ is $\pi_1(S)$-invariant and as
$$\tilde{S}= \bigcup_{\gamma \in \pi_1(S)} \gamma(D),$$
we obtain that $C'_0= \max_{\tilde{x} \in \tilde{S}} d(\tilde{f}(\tilde{x}),\tilde{A}(\tilde{x})).$
In the same way, we can prove that $C''_0= \max_{\tilde{x} \in \tilde{S}} d(\tilde{f}^{-1}(\tilde{x}),\tilde{A}^{-1}(\tilde{x}))$ exists. It suffices then to take $C_0=\max(C'_0,C''_0)$. 
\end{proof}

Observe that $L_C$ is a union of leaves of $\tilde{\mathcal{F}}^s$ and that the closure $\overline{L}_C$ of $L_C$ in $\overline{S}$ is a compact connected subset of $\overline{S}$.

\begin{lemma} \label{Lem:bigenoughnbhd}
For any $C> \frac{C_0}{1-\lambda}$, there exists $\frac{C_0}{1-\lambda}<C'<C$ such that, for any leaf $L \in \tilde{\mathcal{F}}^s$, the following inclusion holds.
$$ \tilde{f}^{-1} \left( \tilde{A}(L)_{C} \right) \subset L_{C'}.$$
\end{lemma}

\begin{proof}
Fix a leaf $L \in \tilde{\mathcal{F}}^s$ and take $\tilde{x} \in \tilde{A}(L)_C$. Then there exists $\tilde{y} \in \tilde{A}(L)$ such that $d_u(\tilde{x},\tilde{y}) \leq C$. Then
$$\begin{array}{rcl}
d_{u}(\tilde{f}^{-1}(\tilde{x}),\tilde{A}^{-1}(\tilde{y})) & \leq  & C_0 + d_u(\tilde{A}^{-1}(\tilde{x}),\tilde{A}^{-1}(\tilde{y})) \\
 & \leq & C_0 +\lambda d_{u}(\tilde{x},\tilde{y}) \\
 & \leq & C_0+\lambda C.
 \end{array}
 $$ 
By taking $C > \frac{C_0}{1-\lambda}$, we obtain that $C_0+ \lambda C <C$. We take then $C'$ such that $C_0+ \lambda C< C'<C$. With those choices, for any leaf $L \in \mathcal{L}^s$ and for any point $\tilde{x} \in \tilde{A}(L)_C$, the point $\tilde{f}^{-1}(\tilde{x})$ belongs to $L_{C'}$.
\end{proof}

For the rest of the proof, we fix $C>C'>\frac{C_0}{1-\lambda}$ given by Lemma \ref{Lem:bigenoughnbhd}. Fix a Cauchy sequence $(L_n)_n$ of elements of $\mathcal{L}^s$ such that, for any $k \geq n$, $d_u(L_n,L_{k}) \leq \lambda^{n} (C-C')$ so that $d_u(\tilde{A}^{n}(L_n),\tilde{A}^{n}(L_{k})) \leq C-C'$. We say that such a Cauchy sequence is \emph{adapted}.

By Lemma \ref{Lem:bigenoughnbhd}, for any $n \geq 0$,
$$ \tilde{f}^{-1} \left( \tilde{A}^{n+1}(L_{n+1})_C \right) \subset \tilde{A}^{n}(L_{n+1})_{C'} \subset \tilde{A}^{n}(L_n)_C$$
so that
$$ \tilde{f}^{-n-1} \left( \tilde{A}^{n+1}(L_{n+1})_C \right) \subset \tilde{f}^{-n} \left(\tilde{A}^{n}(L_n)_C \right).$$
We define
$$\begin{array}{rcl} \overline{\eta}^{s}((L_n)_{n \geq 0}) & = & \displaystyle \bigcap_{n \geq 0} \overline{\tilde{f}^{-n} \left( \tilde{A}^n(L_n)_C \right)} \\
 & = & \displaystyle \bigcap_{n \in \mathbb{Z}} \overline{\tilde{f}^{-n} \left( \tilde{A}^n(L_n)_C \right)},
 \end{array}$$

where we set $L_n=L_0$ for $n<0$
The subset $\overline{\eta}^{s} \left( (L_n)_{n \geq 0} \right)$ is the intersection of a nested sequence of non empty compact connected subsets of $\overline{S}$. Hence $\overline{\eta}^{s}(L)$ is a nonempty compact and connected subset of $\overline{S}$.

Fix two adapted Cauchy sequences $(L_n)_{n \geq 0}$ and $(L'_n)_{n \geq 0}$ of elements of $\mathcal{L}^s$ which are equivalent, and let us check that $\overline{\eta}^{s}((L_n)_{n})=\overline{\eta}^{s}((L'_n)_{n})$. First, observe that, when the sequence $(L'_k)_k=(L_{n_k})_k$ is a subsequence of $(L_n)$, then it is adapted. Moreover, we have, for any $k$,
$$\tilde{f}^{-k-1} \left( \tilde{A}^{k+1}(L_{n_{k+1}})_C \right)\subset \tilde{f}^{-k} \left( \tilde{A}^{k}(L_{n_{k+1}})_{C'}\right) \subset \tilde{f}^{-k} \left( \tilde{A}^{k}(L_{k})_{C}\right),$$
as $k \leq n_{k+1}$ and $d_u(A^{k}(L_k),A^{k}(L_{n_{k+1}})) \leq C-C'$. Hence, taking intersections,
$$\overline{\eta}^{s}((L_{n_k})_{k \geq 0}) \subset \overline{\eta}^{s}((L_{n})_{n \geq 0}).$$

For any integer $k$, we can choose an integer $n_k \geq k+1$ so that $d_u(L'_{n_k},L_{n_k}) \leq (\lambda^k-\lambda^{k+1})(C-C').$ so that $d_u(L'_{n_k},L_{k+1}) \leq \lambda^k (C-C')$. Also, the integers $n_k$ can be chosen in such a way that, for any $k$, $n_{k+1}>n_k$. Then, as $d_u(L_{n_k},L_k) \leq \lambda^{k+1}(C-C')$, 
$$ \tilde{f}^{-k-1} \left( \tilde{A}^{k+1}(L_{k+1})_C \right) \subset \tilde{f}^{-k} \left(\tilde{A}^{k}(L_{k+1})_{C'} \right)\subset \tilde{f}^{-k} \left(\tilde{A}^{k}(L'_{n_k})_C \right)$$
so that
$$\overline{\eta}^s \left( (L_n)_{n\geq 0} \right) \subset \overline{\eta}^{s}\left( (L'_{n_k})_k \right) \subset \overline{\eta}^s \left((L'_n)_{n\geq 0}\right).$$
Changing the roles of the sequences $(L_n)_n$ and $(L'_n)_n$, we prove the reverse inclusion. 

We just proved that two equivalent Cauchy sequences of elements of $\mathcal{L}^s$ define the same subset. Observe that, by taking subsequences, any Cauchy sequence of elements of $\mathcal{L}^s$ is equivalent to an adapted one. Hence, for any element $\xi \in \overline{\mathcal{L}}^s$, we define $\overline{\eta}^s(\xi)=\overline{\eta}^s((L_n)_{n \geq 0})$, where $(L_n)_{n \geq 0}$ is any adapted Cauchy sequence in the class $\xi$.

Moreover, for any $\xi \in \overline{\mathcal{L}}^{s}$ which is represented by the adapted Cauchy sequence $(L_n)_n$, and by setting $L_n=L_0$ if $n < 0$,
$$ \begin{array}{rcl}
\tilde{f}^{-1} \left( \overline{\eta}^{s}(\xi) \right) & = & \displaystyle \bigcap_{n \in \mathbb{Z}} \overline{\tilde{f}^{-n-1} \left( \tilde{A}^n(L_n)_C \right)} \\
 & = & \displaystyle \bigcap_{n \in \mathbb{Z}} \overline{\tilde{f}^{-n} \left( \tilde{A}^n(\tilde{A}^{-1}(L_{n-1}))_C \right)} \\
 & = &  \overline{\eta}^{s} \left( \tilde{A}^{-1}(\xi) \right),
\end{array}$$
as the sequence $(\tilde{A}^{-1}(L_{n-1}))_n$ is equivalent to the sequence $(\tilde{A}^{-1}(L_{n}))_n$. Hence point $3.$ of Theorem \ref{Th:generalized} holds for $\overline{\mathcal{L}}^{s}$. As $\tilde{f}_{|\partial \tilde{S}}=\tilde{A}_{|\partial \tilde{S}}$, point 5. of Theorem \ref{Th:generalized} holds.

Let us prove point $4.$ now. Let $\gamma \in \pi_1(S)$, $\xi \in \overline{\mathcal{L}}^{s}$ which is represented by an adapted Cauchy sequence $(L_n)_n$ of leaves and $n \geq 0$. Then there exists a deck transformation $\delta \in \pi_1(S)$ such that
$$\left\{ \begin{array}{rcl}
\tilde{A}^{n} \gamma & = & \delta \tilde{A}^n \\
\tilde{f}^{n} \gamma & = & \delta \tilde{f}^n.
\end{array}
\right.
$$
Therefore, as the pseudo-distance $d_u$ is invariant under $\pi_1(S)$,
$$\begin{array}{rcl}
\tilde{f}^{-n} \left((\tilde{A}^{n}(\gamma L_n))_C \right) & = & \tilde{f}^{-n} \left( \left(\delta\tilde{A}^{n}( L_n) \right)_C \right) \\
& = & \tilde{f}^{-n} \left( \delta(\tilde{A}^{n}( L_n))_C \right) \\
 & = & \gamma \tilde{f}^{-n} \left( (\tilde{A}^{n}( L_n))_C \right).
 \end{array}$$
Hence, by taking the intersection of the closures of these sets over $n \geq 0$, we obtain
$$\overline{\eta}^s \left(\gamma(\xi) \right)=\gamma \left(\overline{\eta}^{s}(\xi) \right).$$

Now, let us prove the first point. We start by showing that, for any $\xi \neq \xi' \in \overline{\mathcal{L}}^s$, the sets $\overline{\eta}^{s}(\xi)$ and $\overline{\eta}^{s}(\xi')$ are disjoint. Fix such $\xi \neq \xi'$. Let $(L_n)$ and $(L'_n)$ be adapted Cauchy sequences which are respectively associated to $\xi$ and $\xi'$. For any $n \geq 0$
$$d_u(L_n,L'_n)=d_{u} \left( \tilde{A}^{-n}(\tilde{A}^{n}(L_n)),\tilde{A}^{-n}(\tilde{A}^{n}(L_n)) \right) = \lambda^{n} d_u \left(\tilde{A}^{n}(L_n),\tilde{A}^{n}(L'_n)\right)$$
so that, as the sequence $(d_u(L_n,L'_n))_n$ converges to a positive real number,
$$ \lim_{n \rightarrow +\infty}d_u \left(\tilde{A}^{n}(L_n),\tilde{A}^{n}(L'_n)\right)=+\infty.$$
Hence there exists $N \geq 0$ such that
$$\tilde{A}^{N}(L_N)_C \cap \tilde{A}^{N}(L'_N)_C= \emptyset.$$
As distinct leaves of $\tilde{\mathcal{F}}^s$ have distinct endpoints on $\partial \tilde{S}$, we deduce that
$$\overline{\tilde{A}^{N}(L_N)_C} \cap \overline{\tilde{A}^{N}(L'_N)_C}= \emptyset.$$
Hence
$$\overline{\eta}^{s}(\xi) \cap \overline{\eta}^{s}(\xi') \subset \overline{\tilde{f}^{-N} \left(\tilde{A}^{N}(L_N)_C \cap \tilde{A}^{N}(L'_N)_C \right)} =\emptyset.$$

Recall that, for any $L \in \mathcal{L}^s$, $\overline{\eta}^{s}(L)\cap \partial \tilde{S}=L \cap \partial \tilde{S}$ consists of a finite number of points, and more than one. Moreover, the set $\overline{\eta}^{s}(L)$ is connected so that $\eta^{s}(L) \neq \emptyset$.

Now, let us prove that, for any point $\tilde{x} \in \tilde{S}$, there exists a point $\xi$ of $\overline{\mathcal{L}}^s$ such that $\tilde{x} \in \eta^{s}(\xi)$, which will prove the first point. 

Fix $n \geq 0$ and $\tilde{x} \in \tilde{S}$. Let us consider the set $E^{n}_{\tilde{x}}$ which is the closure in $\overline{\mathcal{L}}^s$ of the set of the leaves $L$ of $\tilde{\mathcal{F}}^{s}$ such that
$$\tilde{f}^{n}(\tilde{x}) \in  \tilde{A}^n(L)_C.$$
Observe that $E^{n+1}_{\tilde{x}} \subset E^{n}_{\tilde{x}}$. Indeed, if $d_u \left(\tilde{A}^{n+1}(L),\tilde{f}^{n+1}(\tilde{x}) \right) \leq C$ for some stable leaf $L$, then
$$\begin{array}{rcl}
d_u \left(\tilde{A}^{n}(L),\tilde{f}^{n}(\tilde{x}) \right) & = & d_u \left(\tilde{A}^{-1}\tilde{A}^{n+1}(L),\tilde{A}^{-1}\tilde{A}\tilde{f}^{n}(\tilde{x}) \right) \\
 & = & \lambda d_u \left(\tilde{A}^{n+1}(L),\tilde{A}\tilde{f}^{n}(\tilde{x}) \right) \\
 & \leq & \lambda \left( d_u \left(\tilde{A}^{n+1}(L),\tilde{f}^{n+1}(\tilde{x}) \right) +C_0 \right) \\
 & \leq & \lambda(C+C_0) \leq C
\end{array}
$$  
as $C$ was chosen so that $C\geq \frac{C_0}{1-\lambda} \geq \frac{\lambda C_0}{1-\lambda}$. We obtain then the wanted inclusion by taking closures in $\overline{\mathcal{L}}^s$.

Let us prove that the diameter of the set $E^n_{\tilde{x}}$ tends to $0$ as $n$ tends to $+\infty$. Indeed, for any two stable leaves $L$ and $L'$ in $E^n_{\tilde{x}}$,
$$d_u(A^{n}(L),A^{n}(L')) \leq 2C$$
so that
$$d_u(L,L') \leq 2C \lambda^{n}.$$
The diameter of the set $E^n_{\tilde{x}}$ is smaller than $2C \lambda^{n}$ and tends to $0$.

Therefore, the set $\displaystyle \bigcap_{n \geq 0} E^{n}_{\tilde{x}}$ is a decreasing sequence of nonempty closed subsets of $\overline{\mathcal{L}}^s$ whose diameters tend to $0$. As the latter set is complete, this intersection consists of a single point $\xi$. Any sequence $(L_n)$ of stable leaves which satisfies that, for any $n$, $L_n$ belongs to $E^{n}_{\tilde{x}}$, is a Cauchy sequence which represents $\xi$. Fix such a sequence $(L_n)$. By definition of the sets $E^n_{\tilde{x}}$, 
$$ \tilde{x} \in \bigcap_{n \geq 0} \tilde{f}^{-n}\left(\tilde{A}^n(L_n)_C \right).$$
Take a subsequence $(L_{n_k})_{k \geq 0}$ of $(L_n)_n$ which is adapted. By Lemma \ref{Lem:bigenoughnbhd}, for any $k\geq 0$, as $n_k \geq k$
$$\tilde{f}^{-n_k}\left(\tilde{A}^{n_k}(L_{n_k})_C\right) \subset \tilde{f}^{-k}\left(\tilde{A}^{k}(L_{n_k})_C\right).$$
Hence
$$ \tilde{x} \in \bigcap_{k \geq 0} \tilde{f}^{-k}\left(\tilde{A}^k(L_{n_k})_C \right)=\eta^s(\xi).$$

The first point of Theorem \ref{Th:generalized} is proved.

Let us prove the second point now. If $E$ is a closed and connected subset of $\tilde{S}$ which is a union of leaves in $\mathcal{L}^s$, denote by $\overline{\mathcal{L}}^s_E$ the subset of $\overline{\mathcal{L}}^s$ consisting of points which are represented by Cauchy sequences of leaves which are contained in $E$. Let
$$\overline{\eta}^{s}(E)=\bigcap_{n \geq 0} \overline{\tilde{f}^{-n}\left(  \tilde{A}^n(E)_{C} \right)}.$$
The subset $\overline{\eta}^{s}(E)$ is compact and connected as a decreasing intersection of compact and connected subsets of $\overline{S}$.

Observe that
$$\overline{\eta}^{s}(E)\cap \tilde{S}=\bigcup_{\xi \in \overline{\mathcal{L}}^s_E} \eta^{s}(\xi).$$
Indeed, suppose that the point $\tilde{x}$ belongs to $\overline{\eta}^{s}(E)\cap \tilde{S}$. Then, by definition, for any $n \geq 0$, there exists a stable leaf $L_n \subset E$ such that $\tilde{f}^{n}(\tilde{x}) \in \tilde{A}^n(L_n)_C$. As we saw in the proof of the first point, the sequence $(L_n)_{n \geq 0}$ has to be a Cauchy sequence and $\tilde{x} \in \eta^s(\xi)$, where $\xi$ is the class of the sequence $(L_n)_n$. The reverse inclusion is clear.
Hence
$$\overline{\eta}^{s}(E)= \overline{\bigcup_{\xi \in \overline{\mathcal{L}}^s_E} \eta^{s}(\xi)}.$$

Suppose now that $E$ is an open and connected subset of $\tilde{S}$ which is a union of leaves of $\tilde{\mathcal{F}}^s$. We write $E=\cup_{n\geq 0} E_n$, where $(E_n)$ is defined as follows. Fix a leaf $L_E$ which is contained in $E$ and define $E_n$ as the connected component of $L_E$ in the closure of the set of points $\tilde{x} \in \tilde{S}$ such that $d_u(\tilde{x}, E^{c}) \geq \min \left(\frac{1}{n},d_u(L_E, E^{c})/2 \right)$. The set
$$\overline{\eta}^{s}(E)=\bigcup_{n \geq 0} \overline{\eta}^{s}(E_n)$$
is connected as an increasing sequence of connected subsets. The definition of $\overline{\eta}^{s}(E)$ does not depend on the chosen leaf $L_E$.

Fix a leaf $L_0$ of $\tilde{\mathcal{F}}^s$. Let $\kappa$ be the finite set consisting of connected components of $\tilde{S} \setminus L_0$. Observe that
$$\tilde{S} \setminus L_0= \bigcup_{E \in \kappa} E.$$
Then, as the subsets $\overline{\eta}^{s}(\xi)$, for $\xi \in \overline{\mathcal{L}}^s$, cover $\tilde{S}$ and are pairwise disjoint,
$$ \tilde{S} \subset A= \overline{\eta}^{s}(L_0) \cup  \bigcup_{E \in \kappa} \overline{\eta}^{s}(E) \subset \overline{S}$$
and the sets appearing in this decomposition are pairwise disjoint. Indeed, as the sets of the form $\eta^{s}(\xi)$ are disjoint for different $\xi$ the intersection of two of those sets with $\tilde{S}$ is empty. Moreover, 
$$\overline{\eta}^{s}(L_0) \cap \partial \tilde{S}=\overline{L}_0 \cap \partial \tilde{S}$$
and, for any $E \in \kappa$ written as union of the sequence $(E_n)$ of closed subsets which are union of leaves,
$$\overline{\eta}^{s}(E) \cap \partial \tilde{S}= \bigcup_{n \geq 0} \overline{E_n} \cap \partial \tilde{S}$$
have to be pairwise disjoint (otherwise $L_0$ would have to belong to one of those sets $E_n$, a contradiction). Morover, as ends of leaves of $\tilde{\mathcal{F}}^s$ are dense in $\partial \tilde{S}$, the set $\bigcup_{n \geq 0} \overline{E_n} \cap \partial \tilde{S}$ is a connected component of $\partial \tilde{S} \setminus \overline{L}_0$ so that
$$\overline{\eta}^{s}(L_0) \cup  \bigcup_{E \in \kappa} \overline{\eta}^{s}(E)=\overline{S}.$$

For any $E \in \kappa$, let $c(E)$ be the closure of $E$ in $\tilde{S}$. Observe that $c(E)=E \cup L_0$. Let us prove also that
$$\overline{\eta}^{s}(c(E))=\overline{\eta}^s(E) \cup \overline{\eta}^s(L_0).$$ 
The reverse inclusion is clear. Fix a decomposition $E= \cup_n E_n$ as above. For the direct inclusion, first, 
$$\overline{\eta}^{s}(c(E))\cap \partial \tilde{S}=\overline{c(E)}\cap \partial \tilde{S}=\overline{E} \cap \partial \tilde{S}=(\overline{\eta}^s(E) \cup \overline{\eta}^s(L_0)) \cap \partial \tilde{S}.$$
Now, any point in $\overline{\eta}^{s}(c(E)) \cap \tilde{S}$ belongs to some $\eta^{s}(\xi)$, for some $\xi \in \mathcal{L}^s_{c(E)}$. Take a Cauchy sequence $(L'_n)$ of leaves which represents $\xi$. Then either the sequence $(d_u(L'_n,L_0))_n$ has to be eventually bounded below and hence $(L'_n)$ is eventually contained in some $E_k$ or the sequence $(L'_n)$ is equivalent to the constant sequence $(L_0)$. In this second case, $\eta^{s}(\xi)=\eta^{s}(L_0)$. This proves the direct inclusion.
 
 Then, for any $E \in \kappa$, the set
$$\overline{S} \setminus \overline{\eta}^{s}(E)= \bigcup_{E'\in \kappa, \ E' \neq E} \overline{\eta}^{s}(c(E'))$$
is closed as a finite union of closed subsets, so that each $\overline{\eta}^{s}(E)$, for $E \in \kappa$, is open in $\overline{S}$. Hence each $\overline{\eta}^{s}(E)$, for $E \in \kappa$, has to be a connected component of $\overline{S} \setminus \overline{\eta}^{s}(L_0)$. 

Now, let $\xi \in \overline{\mathcal{L}}^s \setminus \mathcal{L}^s$. Fix a leaf $L_0 \in \mathcal{L}^s$. In this case, we define, for any $n \geq 1$, $E_n$ as the connected component of $L_0$ in the set of leaves $L$ of $\mathcal{L}^s$ such that 
$$d_u(L,\xi) \geq \min \left(\frac{d_u(L_0,\xi)}{2} ,\frac{1}{n} \right).$$
 Then the union $A$ of the $\overline{\eta}^s(E_n)$, for $n\geq 1$, is connected as an increasing union of connected subsets. Moreover, $$A \cup \overline{\eta}^{s}(\xi)= \overline{S}$$ and the subsets $A$ and $\overline{\eta}^{s}(\xi)$ are disjoint. Hence the complement of the subset $\overline{\eta}^{s}(\xi)$, which is $A$, is connected.

We move to the proof of point 6. of Theorem \ref{Th:generalized}.

For this point, we need the following claim, where the constant $C_0$ is given by Claim \ref{uniformbound}.

\begin{claim}
For any two points $\tilde{x}$ and $\tilde{y}$ and any $n\geq 0$,
$$d_s(\tilde{f}^{n}(\tilde{x}), \tilde{f}^n(\tilde{y})) \leq \frac{2C_0}{1- \lambda}+d_s(\tilde{x},\tilde{y}).$$
\end{claim}

\begin{proof}
Fix $n \geq 0$.
$$\begin{array}{rcl}
d_s(\tilde{f}^{n+1}(\tilde{x}),\tilde{f}^{n+1}(\tilde{y})) & \leq & d_s(\tilde{f}^{n+1}(\tilde{x}),\tilde{A}\tilde{f}^{n}(\tilde{x}))+d_s(\tilde{A}\tilde{f}^{n}(\tilde{x}),\tilde{A}\tilde{f}^{n}(\tilde{y}))+d_s(\tilde{A}\tilde{f}^{n}(\tilde{y}),\tilde{f}^{n+1}(\tilde{y})) \\
 & \leq & C_0+ \lambda d_s(\tilde{f}^{n}(\tilde{x}),\tilde{f}^{n}(\tilde{y}))+C_0.
\end{array}$$
Then, an induction implies that, for any $n \geq 0$,
$$d_s(\tilde{f}^{n}(\tilde{x}),\tilde{f}^{n}(\tilde{y})) \leq 2C_0 \left( \sum_{i=0}^{n-1} \lambda^i \right) + \lambda^n d_s(\tilde{x},\tilde{y}).$$
As 
$$\sum_{i=0}^{n-1} \lambda^i \leq \frac{1}{1-\lambda}=\sum_{i=0}^{+\infty} \lambda^i,$$
this implies the claim.   
\end{proof}

Fix a point $\xi \in \overline{\mathcal{L}}^s$ which is represented by an adapted Cauchy sequence $(L_n)$ of leaves. By construction of $\eta^{s}(\xi)$, for any $n \geq 0$ and any $\tilde{y} \in \eta^{s}(\xi)$,
$$d_u(\tilde{f}^{n}(\tilde{y}),\tilde{A}^{n}(L_n)) \leq C,$$
where $C>0$ is given by Lemma \ref{Lem:bigenoughnbhd}. Hence, for any pair of points $\tilde{x}$ and $\tilde{y}$ in $\eta^s(\xi)$,
$$d_{u}(\tilde{f}^{n}(\tilde{x}),\tilde{f}^{n}(\tilde{y})) \leq 2C$$
and
$$\begin{array}{rcl}
d(\tilde{f}^{n}(\tilde{x}),\tilde{f}^{n}(\tilde{y})) & \leq & 2C+d_s(\tilde{f}^{n}(\tilde{x}),\tilde{f}^{n}(\tilde{y})) \\
 & \leq & 2C+ \frac{2C_0}{1- \lambda}+d_s(\tilde{x},\tilde{y})
 \end{array}$$
by the above claim. This proves the direct inclusion. 

To prove the reverse inclusion, we will actually prove that the set $\tilde{S} \setminus \eta^{s}(\xi)$ is contained in the complement of the right hand set in the statement.
Let $\tilde{y} \in \tilde{S} \setminus \eta^{s}(\xi)$ and take $\xi' \in \mathcal{L}^s$ such that $\tilde{y} \in \eta^{s}(\xi')$. Let $(L'_n)$ be an adapted Cauchy sequence of leaves which represents $\xi'$. Then, for any point $\tilde{x} \in \eta^s(\xi)$, 
$$\begin{array}{rcl}
d_u(\tilde{f}^{n}(\tilde{x}),\tilde{f}^n(\tilde{y})) & \geq & d_{u}(\tilde{A}^{n}(L_n),\tilde{A}^{n} (L'_n))-d_u(\tilde{A}^{n}(L_n),\tilde{f}^{n}(\tilde{x}))-d_u(\tilde{A}^{n}(L'_n),\tilde{f}^{n}(\tilde{y})) \\
 & \geq & \lambda^{-n}d_u(L_n,L'_n)-2C \rightarrow +\infty.
 \end{array}$$
 which proves the wanted inclusion.

Now, let us prove point 7. Fix a point $\xi \in \overline{\mathcal{L}^s}$.



We will use here point 6. which implies that the definition of $\overline{\eta}^s(\xi)$ does not depend on the chosen $C>0$. Fix $C>\frac{C_0}{1-\lambda}$, a neighbourhood $U$ of $\overline{\eta}^s(\xi)$ in $\overline{S}$ and an adapted Cauchy sequence $(L_n)$ of stable leaves which represents $\xi$.

Recall that
$$\overline{\eta}^{s}(\xi)=\bigcap_{n \geq 0} \tilde{f}^{-n}\left(  \overline{\tilde{A}^n(L_n)_{C}} \right).$$

Then, as the subset $\overline{\eta}^{s}(\xi)$ is the intersection of a nested sequence of compact subsets of $\overline{S}$, there exists $N>0$ such that

$$\tilde{f}^{-N}\left( \overline{ \tilde{A}^N(L_N)_{C}} \right)=\bigcap_{n = 0}^N \tilde{f}^{-n}\left( \overline{ \tilde{A}^n(L_n)_{C}} \right) \subset U.$$

Observe that, by Lemma \ref{Lem:bigenoughnbhd} and the definition of an adapted sequence, there exists $C'>\frac{C_0}{1-\lambda}$ such that, for any $n > N$,

$$\overline{\eta}^s(\xi) \subset \tilde{f}^{-n}\left( \overline{ \tilde{A}^n(L_n)_{C}} \right) \subset \tilde{f}^{-N}\left( \overline{ \tilde{A}^N(L_n)_{C'}} \right) \subset \tilde{f}^{-N}\left( \overline{ \tilde{A}^N(L_N)_{C}} \right) \subset U.$$ 

Now, fix a constant $D$ with $\frac{C_0}{1-\lambda}<D<C'$. Applying Lemma \ref{Lem:bigenoughnbhd} with this constant $D$ gives a new constant $D>D'>\frac{C_0}{1-\lambda}$ such that the lemma holds. Take a point $\xi' \in \overline{\mathcal{L}}^s$ such that $d_u(\xi',\xi) < \lambda^N( C'-D')$. Take an adapted Cauchy sequence $(L'_n)$ of stable leaves for the constants $D$ and $D'$ which represents $\xi'$. Then, for any large enough $n$, $d_u(L_n,L'_n) < \lambda^N(C'-D')$. Fix such an integer $n > N$. Then $d_u(\tilde{A}^N(L_n), \tilde{A}^N(L'_n))<C'-D'$ so that
$$\tilde{A}^N(L'_n)_{D'} \subset \tilde{A}^N(L_n)_{C'}$$
and
$$\overline{\eta}^s(\xi') \subset \tilde{f}^{-N} \left( \overline{ \tilde{A}^N(L'_n)_{D'}} \right) \subset \tilde{f}^{-N}\left( \overline{ \tilde{A}^N(L_n)_{C'}} \right) \subset U.$$

This proves point 7. and completes the proof of Theorem \ref{Th:generalized}.
\end{proof}

\section{Extension to the case of surfaces with marked points}

Let $S$ be a closed surface and $F$ be a finite subset of $S$. In this section, we fix a homeomorphism $A$ of $S$ which is pseudo-Anosov relative to $F$. The homeomorphism $A$ preserves two transverse singular foliations $\mathcal{F}^s$ and $\mathcal{F}^u$. The only difference with the case of an actual pseudo-Anosov homeomorphism of $S$ is that these foliations can have one-pronged singularities. Let us denote by $\tilde{\mathcal{F}}^s$ and $\tilde{\mathcal{F}}^u$ the respective lifts of $\mathcal{F}^s$ and $\mathcal{F}^u$ to the universal cover $\tilde{S}$ of $S$. As before, we denote by $\overline{\mathcal{L}}^s$ (respectively $\overline{\mathcal{L}}^u$) the completion of the space of stable (resp. unstable) leaves of $\tilde{A}$.

Let $f$ be a homeomorphism which is isotopic to $A$. Take an isotopy $(f_t)$ with $f_1=f$ and $f_0=A$ which pointwise fixes $F$ and lift it to an isotopy $(\tilde{f}_t)_{t \in [0,1]}$ of $\tilde{S}$ with $\tilde{f}_0=\tilde{A}$. Then $\tilde{f}=\tilde{f}_1:\tilde{S} \rightarrow \tilde{S}$ is a lift of $f$.

The following theorem is a generalization of Theorem \ref{Th:generalized} to this setting. However, in this statement, we do not mention any behaviour of the generalized leaves with respect to the boundary at infinity of $\tilde{S}$, as, in this case, we do not control a priori the behaviour at infinity of the leaves of $\tilde{\mathcal{F}}^s$ and $\tilde{\mathcal{F}}^u$. However, all the other results are still true.

\begin{theorem} \label{Th:generalizedfurther}
Fix $i=s,u$. To each point $\xi$ of $\overline{\mathcal{L}}^i$, we associate a nonempty connected subset $\eta^i(\xi)$ of $\tilde{S}$ with the following properties.
\begin{enumerate}
\item The $\eta^{i}(\xi)$, for $\xi \in \overline{\mathcal{L}}^i$, form a partition of $\tilde{S}$.
\item For every $L \in \mathcal{L}^i$, the set $\tilde{S} \setminus \eta^{i}(L)$ has the same number of connected components as $\tilde{S} \setminus L$.
\item For every $\xi \in \overline{\mathcal{L}}^i$,
$$ \tilde{f} \left( \eta^i(\xi) \right)=\eta^i \left(A(\xi) \right).$$
\item For every $\xi \in \overline{\mathcal{L}}^i$ and any deck transformation $\gamma \in \pi_1(S)$,
$$ \gamma \left( \eta^i(\xi) \right)=\eta^i \left( \gamma(\xi) \right).$$
\item The sets $\eta^s(\xi)$ are the stable sets of all its points, namely for every $\xi \in \overline{\mathcal{L}}^s$ and $\tilde{x} \in \eta^{s}(\xi)$,
$$ \eta^{s}(\xi)= \left\{ \tilde{y} \in \tilde{S} \ | \ \exists C'>0, \ \forall n \geq 0, d(\tilde{f}^{n}(\tilde{x}),\tilde{f}^{n}(\tilde{y})) \leq C' \right\}.$$
The sets $\eta^u(\xi)$ are the unstable sets of all its points, namely for every $\xi \in \overline{\mathcal{L}}^u$ and $\tilde{x} \in \eta^{u}(L)$,
$$ \eta^{u}(L)= \left\{ \tilde{y} \in \tilde{S} \ | \ \exists C'>0, \ \forall n \geq 0, d(\tilde{f}^{-n}(\tilde{x}),\tilde{f}^{-n}(\tilde{y})) \leq C' \right\}.$$
\end{enumerate}
\end{theorem}

We will prove this theorem by using Theorem \ref{Th:generalized} and a finite branched covering map. This covering trick was indicated to me by Sebastian Hensel. 

\begin{remark}
By construction, the subsets $\eta^u(\xi) \cap \eta^s(\xi')$ have a diameter which is uniformly bounded, which makes the connection with the recent preprint by Alejo García-Sassi and Fábio Armando Tal \cite{GT}.
\end{remark}

\begin{proof}
Denote by $\sigma\subset F$ the subset consisting of one-pronged singular points of the foliation $\mathcal{F}^s$ (which is also the set of one-pronged singular points of $\mathcal{F}^u$) and fix a point $p \in S \setminus F$. For each such singularity $s \in \sigma$ choose a simple loop based at $p$ which bounds a once punctured disk of $S \setminus F$, with puncture at $s$. Let $E$ be the finite subset of $\pi_1(S\setminus F,p)$ consisting of such loops. As any free group is residually finite, there exists a finite index normal subgroup of $\pi_1(S\setminus F,p)$ which does not contain any point of $E$. This subgroup in turn defines a finite cover of $S \setminus F$ and a branched cover $b:\hat{S}\rightarrow S$.

The homeomorphisms $f$ and $A$ lift to isotopic homeomorphisms $\hat{f}$ and $\hat{A}$ of $\hat{S}$. Observe that $\hat{A}$ is now a pseudo-Anosov homeomorphism relative to $b^{-1}(F)$ with no one-pronged singularity points at points of $p^{-1}(F)$ so that $\hat{A}$ is a pseudo-Anosov homeomorphism of $\hat{S}$. Theorem \ref{Th:generalized} yields a stable and an unstable partition of $\hat{S}$ which correspond to $\hat{f}$. By construction of this partition (see the proof of Theorem \ref{Th:generalized}), this partition is invariant under the deck transformations of $b$. Hence this partition projects to a stable and an unstable partition of $S$ for $f$ which in turn lift to $\tilde{S}$ to give the wanted partitions.
\end{proof}

\small

\bibliographystyle{amsalpha}
\bibliography{Biblio}

\end{document}